\documentclass[12pt, reqno]{amsart}
\usepackage{bm, fullpage}

\newtheorem{thm}{Theorem}[section]

\newtheorem{lemma}[thm]{Lemma}
\newtheorem{theorem}[thm]{Theorem}

\numberwithin{equation}{section}

\theoremstyle{definition}
\newtheorem{remark}[thm]{Remark}

\renewcommand{\qed}{\rule{3mm}{3mm}}
\renewenvironment{proof}
    {\vspace{1mm}\noindent\textbf{Proof.}}
    {\hspace*{\fill} $\qed$\vspace{1mm}}
\newenvironment{proof_of}[1]
    {\vspace{1mm}\noindent {\bf Proof of #1.}}
    {\hspace*{\fill} $\qed$\vspace{1mm}}

\newcommand{\R}{\mathbb R}
\newcommand{\al}{\alpha}
\newcommand{\be}{\beta}
\newcommand{\la}{\lambda}
\renewcommand{\phi}{\varphi}
\newcommand{\e}{\varepsilon}
\DeclareMathOperator{\ddiv}{div}

\begin{document}
\title[Stability of positive radial steady states]
{Stability of positive radial steady states for the parabolic H\'enon-Lane-Emden system}

\author{Daniel Devine} \email{dadevine@tcd.ie}
\author{Paschalis Karageorgis} \email{pete@maths.tcd.ie}
\address{School of Mathematics, Trinity College Dublin, Ireland.}


\begin{abstract}
When it comes to the nonlinear heat equation $u_t - \Delta u = u^p$, the stability of positive
radial steady states in the supercritical case was established in the classical paper by Gui, Ni
and Wang.  We extend this result to systems of reaction-diffusion equations by studying the
positive radial steady states of the parabolic H\'enon-Lane-Emden system
\begin{equation*}
\left\{ \begin{aligned}
u_t - \Delta u &= |x|^k v^p &\mbox{ in } \R^n \times (0,\infty),\\
v_t - \Delta v &= |x|^l u^q &\mbox{ in } \R^n \times (0,\infty),
\end{aligned} \right.
\end{equation*}
where $k,l\geq 0$, $p,q\geq 1$ and $pq>1$.  Assume that $(p,q)$ lies either on or above the
Joseph-Lundgren critical curve which arose in the work of Chen, Dupaigne and Ghergu.  Then all
positive radial steady states have the same asymptotic behavior at infinity, and they are all
stable solutions of the parabolic H\'enon-Lane-Emden system in $\R^n$.
\end{abstract}

\maketitle

\section{Introduction}
We study the positive radial steady states of the parabolic H\'enon-Lane-Emden system
\begin{equation} \label{psys}
\left\{ \begin{aligned}
u_t - \Delta u &= |x|^k v^p &\mbox{ in } \R^n \times (0,\infty), \\
v_t - \Delta v &= |x|^l u^q &\mbox{ in } \R^n \times (0,\infty),
\end{aligned} \right.
\end{equation}
where $k,l\geq 0$, $p,q\geq 1$ and $pq>1$.  When it comes to the scalar equation $u_t - \Delta u =
u^p$, the classical paper by Gui, Ni and Wang \cite{GNW} gave rise to a critical power $p_c$ which
determines the behavior of all positive radial steady states.  If $1<p<p_c$, they are unstable, and
the graphs of any two steady states intersect one another.  If $p\geq p_c$, they are stable, and
the graphs of distinct steady states do not intersect.  In this paper, we establish an analogue of
the latter statement for the parabolic system \eqref{psys}, and we also complement the results of
Chen, Dupaigne and Ghergu \cite{CDG} for the corresponding elliptic system.

The positive steady states of \eqref{psys} satisfy the elliptic H\'enon-Lane-Emden system
\begin{equation} \label{esys}
\left\{ \begin{aligned}
-\Delta u &= |x|^k v^p &\mbox{ in } \R^n, \\
-\Delta v &= |x|^l u^q &\mbox{ in } \R^n,
\end{aligned} \right.
\end{equation}
which has been extensively studied over the last three decades.  Starting with the seminal papers
by Mitidieri \cite{EM96} and Serrin and Zou \cite{SZ2, SZ1} for the case $k=l=0$, it is now known
that positive regular radial solutions exist if and only if
\begin{equation} \label{Sob}
\frac{k+n}{p+1} + \frac{l+n}{q+1} \leq n-2.
\end{equation}
We refer the reader to \cite{BVG} and the references cited therein.  The same result is expected to
also hold for non-radial solutions, but this part has only been settled in dimensions $n\leq 4$,
and only partial results are available in higher dimensions \cite{BM, PQS, PS}.

Regarding the qualitative properties of positive radial solutions, a key role is played by an
explicit radial singular solution $(u_*,v_*)$ which has the form
\begin{equation} \label{ss}
u_*(r) = C_\al r^{-\al} = C_\al |x|^{-\al}, \qquad
v_*(r) = C_\be r^{-\be} = C_\be |x|^{-\be}.
\end{equation}
Assuming that \eqref{Sob} holds, this is easily seen to satisfy \eqref{esys} at all points $x\neq
0$, if
\begin{equation} \label{AB}
\al = \frac{k+2 + (l+2)p}{pq-1} ,\qquad \be = \frac{l+2 + (k+2)q}{pq-1}
\end{equation}
and the constants $C_\al, C_\be$ are given by
\begin{equation} \label{CaCb}
C_\al^{pq-1} = Q(\al) Q(\be)^p, \qquad C_\be^{pq-1} = Q(\be) Q(\al)^q,
\end{equation}
where $Q(\la)=\la(n-2-\la)$ for all $\la\in\R$.  Note that the existence condition \eqref{Sob} can
also be expressed in the equivalent form $\al+\be \leq n-2$.  Thus, $Q(\al)$ and $Q(\be)$ are both
positive, so the coefficients $C_\al,C_\be$ are well-defined, as long as $pq>1$ and \eqref{Sob}
holds.

A solution of the elliptic system \eqref{esys} is called linearly stable, if the associated
linearized system has a positive supersolution.  The linear stability of regular positive radial
solutions has been addressed by Chen, Dupaigne and Ghergu \cite{CDG}.  If it happens that
\begin{equation} \label{JL}
\left[ \frac{(n-2)^2 - (\al-\be)^2}{4} \right]^2 \geq pq Q(\al) Q(\be),
\end{equation}
then all regular positive radial solutions are linearly stable, and if \eqref{JL} does not hold,
then none of them are linearly stable.  We shall refer to \eqref{JL} as the Joseph-Lundgren
condition.  It is the exact analogue of a condition that arose in the phase plane analysis of
\cite{JL} for the scalar equation $-\Delta u = u^p$, and it only holds (for some $p,q$) in
dimensions $n\geq 11$.

The left hand side of \eqref{JL} arises as the sharp constant in a Rellich-type inequality that was
obtained by Caldiroli and Musina \cite{CM}.  Its relationship with the linear stability of radial
solutions is explained in \cite{CDG} for systems, in \cite{KS} for the equation $-\Delta u = u^p$,
and in \cite{PK09} for the biharmonic equation $\Delta^2 u = u^p$ which corresponds to the case
$p>q=1$.

The Joseph-Lundgren condition is also closely related to the intersection properties of radial
solutions.  Assuming \eqref{JL}, in particular, Chen, Dupaigne and Ghergu \cite{CDG} showed that
every regular positive radial solution $(u,v)$ satisfies
\begin{equation} \label{sep}
u_*(r) > u(r) \quad\text{and}\quad v_*(r) > v(r) \quad\text{for all $r>0$.}
\end{equation}
This is sometimes called the separation property for radial solutions.  It was first derived by
Wang \cite{W93} for the equation $-\Delta u = u^p$, and by the second author \cite{PK09} for the
biharmonic equation $\Delta^2 u = u^p$.  For the two scalar equations, it is already known that the
separation property holds if and only if \eqref{JL} does; see \cite{FGK, W93} for more details.

In this paper, we continue the study of the linearly stable positive radial solutions, and we
complement the results of \cite{CDG} in two possible directions.  First, we focus on the elliptic
system \eqref{esys}, and we determine the asymptotic behavior of all linearly stable positive
radial solutions.  More precisely, we show that they all approach the singular solution \eqref{ss}
and that the convergence occurs monotonically.

\begin{theorem}[Monotonic convergence] \label{MCT}
Consider the elliptic H\'enon-Lane-Emden system \eqref{esys}, where $k,l\geq 0$, $p,q\geq 1$,
$pq>1$ and $n\geq 11$. Assume the existence condition \eqref{Sob} and the Joseph-Lundgren condition
\eqref{JL}. If $(u,v)$ is a regular positive radial solution of \eqref{esys} and $(u_*,v_*)$ is the
singular solution \eqref{ss}, then $u/u_*$ and $v/v_*$ are increasing with
\begin{equation} \label{lim=1}
\lim_{r\to \infty} \frac{u(r)}{u_*(r)} = \lim_{r\to \infty} \frac{v(r)}{v_*(r)} = 1.
\end{equation}
\end{theorem}

Theorem \ref{MCT} was only known in the scalar case.  We refer the reader to \cite{LiYi, W93} for
the equation $-\Delta u = u^p$ and to \cite{FGK, GG1, PK09} for the biharmonic equation $\Delta^2 u
= u^p$.  If \eqref{Sob} holds with equality, then the Joseph-Lundgren condition \eqref{JL} is not
satisfied, so our result is not applicable.  In that exceptional case, the asymptotic behavior of
radial solutions is already known, but it is not dictated by the singular solution; see \cite{BVG,
HV} for more details.

The monotonic convergence of radial solutions for the scalar equation $-\Delta u = u^p$ is one of
the key ingredients in the classical paper by Gui, Ni and Wang \cite{GNW} towards the proof of
nonlinear stability of steady states for the associated parabolic equation $u_t - \Delta u = u^p$.
Following a similar approach, we employ Theorem \ref{MCT} to prove nonlinear stability of steady
states, more generally, for the parabolic H\'enon-Lane-Emden system \eqref{psys}.  As in
\cite{GNW}, we need to find the precise asymptotic expansion for all radial solutions at infinity.
Assuming that strict inequality holds in the Joseph-Lundgren condition \eqref{JL}, we show that
\begin{equation} \label{exp}
\begin{aligned}
u(r) &= C_\al r^{-\al} + D_1 r^{-\al-\gamma} + o(r^{-\al-\gamma})
&\quad\text{as $r\to \infty,$}\\
v(r) &= C_\be r^{-\be} + D_2 r^{-\be-\gamma} + o(r^{-\be-\gamma})
&\quad\text{as $r\to \infty$\phantom{,}}
\end{aligned}
\end{equation}
for some constants $D_1,D_2<0<\gamma$.  This asymptotic expansion was derived in \cite{GNW, LiYi}
for the equation $-\Delta u = u^p$ and in \cite{PK12, MW} for the biharmonic equation $\Delta^2 u =
u^p$. To prove nonlinear stability for the parabolic system \eqref{psys}, we are thus led to
introduce the norm
\begin{equation} \label{Norm1}
||(u,v)||_\gamma = \sup_{x\in\R^n}
\bigl| (1+|x|)^{\al + \gamma} u(x)\bigr| + \bigl| (1+|x|)^{\be + \gamma} v(x) \bigr|.
\end{equation}
If equality holds in \eqref{JL}, the asymptotic expansion \eqref{exp} needs to be slightly modified
to include logarithmic factors; see Theorem \ref{AsEx}.  In that case, a suitable norm is then
\begin{equation} \label{Norm2}
|||(u,v)|||_\gamma = \sup_{x\in\R^n}
\left| \frac{(1+|x|)^{\al + \gamma} u(x)}{\log(2 +|x|)} \right| +
\left| \frac{(1+|x|)^{\be + \gamma} v(x)}{\log(2 +|x|)} \right|.
\end{equation}

\begin{theorem}[Nonlinear stability] \label{ST}
Let the assumptions of the previous theorem hold, and consider the parabolic H\'enon-Lane-Emden
system \eqref{psys}.

\begin{itemize}
\item[(a)]
If strict inequality holds in \eqref{JL}, then $(u,v)$ is a stable solution of \eqref{psys} with
respect to the norm \eqref{Norm1} for some $\gamma>0$.

\item[(b)]
If equality holds in \eqref{JL}, then $(u,v)$ is a stable solution of \eqref{psys} with respect
to the norm \eqref{Norm2} for some $\gamma>0$.
\end{itemize}
In each case, $\gamma>0$ is an explicit constant which is defined in Lemmas \ref{Roots} and
\ref{LinSol}.
\end{theorem}

To the best of our knowledge, the only scalar analogue of Theorem \ref{ST} appeared in the
classical paper by Gui, Ni and Wang \cite{GNW}.  There are other related results, however, which
establish the existence of small-amplitude solutions.  In some sense, these can be regarded as
nonlinear stability results for the zero solution.  We refer the reader to \cite{GG2} for the
scalar biharmonic equation $u_t + \Delta^2 u = |u|^{p-1}u$ and to \cite{EH, LL} for the parabolic
system \eqref{psys}.

Finally, we note that several authors have studied quasilinear variations of the elliptic system
\eqref{esys} that involve the $m$-Laplace operator.  Perhaps the most relevant results are the ones
dealing with the coercive version of our system which reads
\begin{equation} \label{coer}
\left\{ \begin{aligned}
\Delta_m u &= |x|^k u^{p_1} v^{p_2} &\mbox{ in } \R^n \\
\Delta_m v &= |x|^l u^{q_1} v^{q_2} &\mbox{ in } \R^n
\end{aligned} \right.
\end{equation}
with $\Delta_m u = \ddiv(|\nabla u|^{m-2}\nabla u)$ for some $m>1$ and $k,l,p_1,p_2,q_1,q_2\geq 0$.
The positive radial solutions of this system are unbounded and increasing, while their monotonic
convergence was established in \cite{DK} under very weak hypotheses.  For the existence and other
properties of solutions, we refer the reader to \cite{BVG, DD, FV, GGS} and the references cited
therein.

In Section 2, we collect some preliminary results that we need to establish the two main theorems.
In Section 3, we prove our monotonic convergence result, Theorem \ref{MCT}, and we also obtain the
asymptotic expansion of all positive radial solutions; see Theorem \ref{AsEx}.  In Section 4, we
then prove our nonlinear stability result, Theorem \ref{ST}.

\section{Preliminary results}

\begin{lemma} \label{Limit}
Suppose $(u,v)$ is a regular positive radial solution of the elliptic system \eqref{esys}, where
$k,l\geq 0$, $p,q\geq 1$ and $pq>1$.  Then both $u(r)$ and $v(r)$ are decreasing with
\begin{equation} \label{lims}
\lim_{r\to \infty} u(r) = \lim_{r\to \infty} v(r) = 0.
\end{equation}
\end{lemma}

\begin{proof}
Since $(u,v)$ is a regular positive radial solution of \eqref{esys}, $u(r)$ satisfies
\begin{equation*}
-(r^{n-1} u'(r))' = r^{n-1+k} \, v(r)^p > 0 \quad\text{for all $r\geq 0$}
\end{equation*}
and also $u'(0)=0$.  Thus, $u(r)$ is decreasing for all $r\geq 0$, and the first limit in
\eqref{lims} exists.  Let us denote this limit by $L\geq 0$.  Then $u(r) \geq L$ for all $r\geq 0$,
so we have
\begin{equation*}
-(r^{n-1} v'(r))' = r^{n-1+l} \, u(r)^q \geq L^q r^{n-1+l} \quad\text{for all $r\geq 0$.}
\end{equation*}
Integrating this equation twice, we conclude that
\begin{equation*}
-v(r) + v(0) \geq \frac{L^q r^{l+2}}{(n+l)(l+2)} \quad\text{for all $r\geq 0$.}
\end{equation*}
Since $l\geq 0$ and $v(r)$ is positive for all $r\geq 0$, this implies that $L=0$.  In other words,
the first limit in \eqref{lims} is zero. The second limit must also be zero for similar reasons.
\end{proof}

To study the linearized equations associated with the elliptic system \eqref{esys}, we shall need a
basic fact about the associated characteristic polynomial.  This has the form
\begin{equation} \label{char}
F(\la) = Q(\al - \la) Q(\be - \la) - pq Q(\al) Q(\be),
\end{equation}
where $Q(\la) = \la(n-2-\la)$ for each $\la\in\R$.  Since the quadratic $Q(\la)$ is symmetric
around the point $\frac{n-2}{2}$, we see that the quartic $F(\la)$ is symmetric around the point
\begin{equation} \label{lstar}
\la_* = \frac{\al + \be}{2} - \frac{n-2}{2}.
\end{equation}
We now exploit this fact to draw some conclusions about the associated eigenvalues.

\begin{lemma} \label{Roots}
Suppose $p,q\geq 1$ and $pq>1$, $\al,\beta>0$ and $\al+\beta\leq n-2$.  Suppose also that the
Joseph-Lundgren condition \eqref{JL} holds.  Then $\la_*<0$ and the polynomial \eqref{char} has a
root in each of the intervals $(-\infty,2\la_*)$, $(2\la_*,\la_*]$, $[\la_*,0)$ and $(0,\infty)$.
Moreover, $\la_*$ is itself a root if and only if \eqref{JL} holds with equality, in which case
$\la_*$ is actually a double root.
\end{lemma}

\begin{proof}
Since $\al+\be\leq n-2$ by assumption, we have $\la_*\leq 0$ by \eqref{lstar}.  Note that the
quartic polynomial \eqref{char} has leading term $\la^4$, and it is symmetric around $\la_*$ with
\begin{equation*}
\lim_{\la \to \pm \infty} F(\la) = +\infty, \qquad F(2\la_*) = F(0) = (1-pq) Q(\al)Q(\be) < 0
\end{equation*}
because $pq>1$ and $0<\al,\be<n-2$.  In addition, it is easy to check that
\begin{equation*}
Q(\al - \la_*) = Q(\be - \la_*) = \frac{(n-2)^2 - (\al-\be)^2}{4}.
\end{equation*}
Thus, the Joseph-Lundgren condition \eqref{JL} can also be expressed in the form $F(\la_*)\geq 0$.
Since $F(0)<0$ by above, this implies that $\la_*\neq 0$ and thus $\la_*<0$.  It trivially follows
by Bolzano's theorem that $F(\la)$ has a root in each of the intervals
\begin{equation*}
(-\infty,2\la_*), \qquad (2\la_*,\la_*], \qquad [\la_*,0), \qquad (0,\infty).
\end{equation*}
Finally, if $\la_*$ itself happens to be a root, then it must be a double root by symmetry.
\end{proof}

\begin{remark}
If the Joseph-Lundgren condition \eqref{JL} does not hold, then the quartic $F(\la)$ has a
$W$-shaped graph which is symmetric around $\la_*$ with $F(\la_*)<0$.  Thus, its unique local
maximum lies below the $\la$-axis, and $F(\la)$ has two real roots only.
\end{remark}

\begin{lemma} \label{LinSol}
Let the assumptions of the previous lemma hold, and let $-\gamma<0$ denote the largest negative
root of the polynomial \eqref{char}.  Then the linearized system
\begin{equation} \label{LinSys}
\left\{ \begin{array}{cc}
-\Delta \phi = p|x|^k v_*^{p-1} \psi \\
-\Delta \psi = q|x|^l u_*^{q-1} \phi
\end{array} \right. \qquad\text{in $\R^n \setminus \{0\}$}
\end{equation}
which is associated with the singular solution \eqref{ss} is satisfied by the functions
\begin{equation} \label{PhiPsi}
\phi(x) = \frac{p C_\be^{p-1}}{Q(\al+\gamma)} \cdot |x|^{-\al-\gamma}, \qquad
\psi(x) = |x|^{-\be-\gamma}.
\end{equation}
\end{lemma}

\begin{proof}
By Lemma \ref{Roots}, the largest negative root $-\gamma$ lies in the interval $[\la_*,0)$, where
$\la_*$ is defined by \eqref{lstar}.  This is easily seen to imply that $Q(\al+\gamma)>0$ because
\begin{equation} \label{a+c}
\al + \gamma \leq \al - \la_* = \frac{n-2}{2} + \frac{\al-\be}{2} <
\frac{n-2}{2} + \frac{\al+\be}{2} \leq n-2.
\end{equation}
Using the fact that $-\Delta |x|^{-\la} = Q(\la) |x|^{-\la-2}$ for all $\la\in\R$, we now find that
\begin{equation*}
-\Delta \phi = pC_\be^{p-1} \cdot |x|^{-\al-\gamma -2} =
pv_*^{p-1} \cdot |x|^{\be(p-1) - \al - \gamma - 2} = p|x|^k v_*^{p-1} \psi
\end{equation*}
by \eqref{ss}, \eqref{AB} and \eqref{PhiPsi}.  Using the exact same argument, we also have
\begin{equation*}
-\Delta \psi = Q(\beta + \gamma) \cdot |x|^{-\beta-\gamma-2}
= \frac{Q(\al+\gamma) Q(\be +\gamma)}{pq C_\al^{q-1} C_\be^{p-1}} \cdot q|x|^l u_*^{q-1} \phi.
\end{equation*}
On the other hand, $-\gamma$ is a root of the characteristic polynomial \eqref{char}, so
\begin{align*}
0 = F(-\gamma) &= Q(\al + \gamma) Q(\be +\gamma) - pq Q(\al) Q(\be) \\
&= Q(\al+\gamma) Q(\be + \gamma) - pq C_\al^{q-1} C_\be^{p-1}
\end{align*}
by \eqref{CaCb}.  Combining the last two equations, we conclude that $-\Delta \psi = q|x|^l
u_*^{q-1} \phi$.
\end{proof}

The last result in this section plays a key role in our subsequent analysis.  This result is due to
Chen, Dupaigne and Ghergu \cite{CDG}.  The proof given below is only a minor modification of the
original proof, and it is included here mostly for the sake of completeness.

\begin{lemma} \label{CoLe}
Let $(u,v)$ be a regular positive radial solution of the elliptic system \eqref{esys}, where
$k,l\geq 0$, $p,q\geq 1$, $pq>1$ and $n\geq 11$.  Let $(u_*,v_*)$ be the singular solution
\eqref{ss}. Assume the existence condition \eqref{Sob} and the Joseph-Lundgren condition
\eqref{JL}.  Then
\begin{equation}
u_*(r) > u(r) \quad\text{and}\quad v_*(r) > v(r) \quad\text{for all $r>0$.}
\end{equation}
\end{lemma}

\begin{proof}
We consider the difference between the two solutions and let
\begin{equation} \label{UV2}
U(r) = u_*(r) - u(r), \qquad V(r) = v_*(r) - v(r).
\end{equation}
Suppose first that one of $U'(r)$, $V'(r)$ does not have a zero.  We may then assume, without loss
of generality, that $U'(r)<0$ for all $r>0$.  Since $\lim_{r\to \infty} U(r) = 0$ by Lemma
\ref{Limit}, this implies that $U(r)>0$ for all $r>0$.  In particular, $u_*(r) > u(r)$ for all
$r>0$, hence
\begin{equation*}
(r^{n-1} V'(r))' = r^{n-1+l} \, (u(r)^q- u_*(r)^q) < 0 \quad\text{for all $r>0$.}
\end{equation*}
On the other hand, the existence condition \eqref{Sob} ensures that $\be < \al+\be \leq n-2$, so
\begin{equation*}
\lim_{r\to 0^+} r^{n-1} V'(r) = \lim_{r\to 0^+} \be C_\be r^{n-2-\be} = 0.
\end{equation*}
We conclude that $V'(r)<0$ for all $r>0$.  Thus, $V(r)$ decreases to zero by Lemma \ref{Limit}, so
we also have $v_*(r) > v(r)$ for all $r>0$, and the result follows.

Suppose now that both $U'(r)$ and $V'(r)$ vanish at some point.  Letting $r_1$ and $r_2$ denote the
first zeros of $U'(r)$ and $V'(r)$, respectively, we must then have
\begin{equation} \label{U'V'}
U'(r) < 0 \quad\text{in $(0,r_1)$}, \qquad V'(r) < 0 \quad\text{in $(0,r_2)$}
\end{equation}
and also $U'(r_1) = V'(r_2)=0$.  Assume that $r_1\leq r_2$ without loss of generality, and let
\begin{align*}
r_3 &= \inf \{ r>0 : V(r) < 0\}, \\
r_4 &= \inf \{ r>0 : U(r) < 0\}.
\end{align*}
If it happens that $r_1\leq r_3$, then $V(r)>0$ in $(0,r_1)$, and this implies that
\begin{equation*}
(r^{n-1}U'(r))' = r^{n-1+k} \, (v(r)^p - v_*(r)^p) < 0 \quad\text{in $(0,r_1)$,}
\end{equation*}
which forces $r^{n-1}U'(r)$ to be positive in $(0,r_1)$, a contradiction.  We conclude that
$r_3<r_1$.  Interchanging the roles of $U(r)$ and $V(r)$, one similarly finds that $r_4 < r_2$.

Next, we show that $r_4\leq r_1$.  If it happens that $r_1<r_4$, then $r_1<r_4<r_2$. On the other
hand, $V(r)$ is strictly decreasing in $(0,r_2)$ by \eqref{U'V'}, so it is negative in $(r_3,r_2)$
and thus
\begin{equation*}
(r^{n-1}U'(r))' = r^{n-1+k} \, (v(r)^p - v_*(r)^p) > 0 \quad\text{in $(r_3,r_2)$.}
\end{equation*}
Since $r_3 < r_1< r_2$ and $U'(r_1)=0$ by definition, $U(r)$ is strictly increasing on $(r_1,r_2)$.
Noting that this interval contains $r_4$, we find that $U(r_1) < U(r_4) = 0$, which obviously
contradicts the definition of $r_4$.  We must thus have $r_4\leq r_1\leq r_2$ and $r_3 < r_1\leq
r_2$. Since $U(r)$ is strictly decreasing on $(0,r_1)$ and $V(r)$ is strictly decreasing on
$(0,r_2)$, this also implies
\begin{equation} \label{r1}
U(r_1) \leq 0, \qquad U'(r_1)= 0, \qquad V(r_1) < 0, \qquad V'(r_1) \leq 0.
\end{equation}

Finally, we use a convexity argument to finish the proof.  Since $p,q\geq 1$, we have
\begin{equation*}
\left\{ \begin{array}{cc}
-\Delta U = |x|^k \, (v_*^p - v^p) \leq p|x|^k v_*^{p-1} V \\
-\Delta V = |x|^l \, (u_*^q - u^q) \leq q|x|^l u_*^{q-1} U
\end{array} \right. \qquad\text{in $\R^n \setminus \{0\}$.}
\end{equation*}
Let $(\phi,\psi)$ be the positive radial solution of the linearized system which is defined by
\eqref{PhiPsi}.  Since this is a positive solution of \eqref{LinSys}, it follows by the last
equation that
\begin{equation*}
\left\{ \begin{array}{cc}
-\psi \Delta U \leq p|x|^k v_*^{p-1} V\psi = -V\Delta \phi \\
-\phi \Delta V \leq q|x|^l u_*^{q-1} U\phi = -U\Delta \psi
\end{array} \right. \qquad\text{in $\R^n \setminus \{0\}$.}
\end{equation*}
Adding the two inequalities and switching to polar coordinates, we conclude that
\begin{equation} \label{Mono}
\left[ r^{n-1} (U\psi' - \psi U') \right]' + \left[ r^{n-1} (V\phi' - \phi V') \right]' \leq 0
\quad\text{for all $r>0$.}
\end{equation}
To check that the expressions in square brackets both vanish as $r\to 0^+$, we now recall our
definitions \eqref{PhiPsi} and \eqref{UV2} of $\phi,\psi,U,V$.  These are easily seen to imply that
\begin{equation*}
r^{n-1} (U\psi' - \psi U') + r^{n-1} (V\phi' - \phi V') = O(r^{n-2-\al-\be-\gamma})
\quad\text{as $r\to 0^+$}.
\end{equation*}
On the other hand, it follows by Lemma \ref{Roots} and our computation \eqref{a+c} that
\begin{equation*}
\al + \be + \gamma \leq \al + \be - \la_* = \la_* + n-2 < n-2.
\end{equation*}
In view of the last two equations, one may then integrate \eqref{Mono} to arrive at
\begin{equation*}
U\psi' - \psi U' + V\phi' - \phi V' \leq 0 \quad\text{for all $r>0$.}
\end{equation*}
This actually contradicts \eqref{r1} because $\phi,\psi > 0$ and since $\phi',\psi'<0$ for all
$r>0$.
\end{proof}

\section{Asymptotic behavior of regular solutions}

In this section, we study the asymptotic behavior of the regular positive radial solutions of the
elliptic system \eqref{esys}.  Let $(u,v)$ be such a solution, and let $(u_*,v_*)$ be the singular
solution \eqref{ss}.  In order to establish Theorem \ref{MCT}, we need to show that
\begin{equation} \label{toshow}
\lim_{r\to \infty} \frac{u(r)}{u_*(r)} = \lim_{r\to \infty} \frac{v(r)}{v_*(r)} = 1
\end{equation}
and that the convergence occurs monotonically in each case.

\begin{proof_of}{Theorem \ref{MCT}}
We use the logarithmic change of variables
\begin{equation} \label{UV}
\mathcal U(s) = \frac{u(e^s)}{u_*(e^s)} = \frac{u(e^s)}{C_\al e^{-\al s}}, \qquad
\mathcal V(s) = \frac{v(e^s)}{v_*(e^s)} = \frac{v(e^s)}{C_\be e^{-\be s}},
\end{equation}
where $s= \log r$.  This is the standard Emden-Fowler transformation that has been used by several
authors in the past \cite{BN, GG1, GNW, PK09}, although the constants $C_\al, C_\be$ are usually
omitted for simplicity.  We have chosen to include them here in order to prove \eqref{toshow}
directly.

First of all, we note that $0 < \mathcal U(s), \mathcal V(s) < 1$ by Lemma \ref{CoLe}. Since
$(u,v)$ is a positive radial solution of the elliptic system \eqref{esys}, it is easy to check that
$(\mathcal U, \mathcal V)$ satisfies
\begin{equation} \label{EF}
\left\{ \begin{aligned}
\mathcal U''(s) + (n-2-2\al) \, \mathcal U'(s) &= Q(\al) \, (\mathcal U(s) - \mathcal V(s)^p) \\
\mathcal V''(s) + (n-2-2\be) \, \mathcal V'(s) &= Q(\be) \, (\mathcal V(s) - \mathcal U(s)^q)
\end{aligned} \right.
\end{equation}
for all $s\in\R$.  To study this autonomous system, we consider the sets
\begin{equation} \label{SUSV}
S_{\mathcal U} = \{ s\in\R : \mathcal U(s) > \mathcal V(s)^p \}, \qquad
S_{\mathcal V} = \{ s\in\R : \mathcal V(s) > \mathcal U(s)^q \}.
\end{equation}
If there exists a point $s_0\notin S_{\mathcal U} \cup S_{\mathcal V}$, then $\mathcal U(s_0) \leq
\mathcal V(s_0)^p \leq \mathcal U(s_0)^{pq}$, and this leads to the contradiction $\mathcal U(s_0)
\geq 1$.  In particular, no such point exists and $S_{\mathcal U} \cup S_{\mathcal V} = \R$.

Next, we claim that $\mathcal U'(s), \mathcal V'(s) \geq 0$ for all $s\in \R$.  This is true as
$s\to -\infty$ because
\begin{equation*}
\lim_{s\to -\infty} e^{-\al s} \, \mathcal U'(s) =
\lim_{r\to 0^+} r^{1-\al} \left( \frac{1}{C_\al} \,r^\al u(r) \right)' = \frac{\al u(0)}{C_\al} > 0
\end{equation*}
and since a similar statement holds for $\mathcal V'(s)$.  Suppose that one of these functions
becomes negative at some point.  We may then assume, without loss of generality, that $\mathcal
U'(s)$ is the first one to become negative.  In that case, there exist $s_1\in\R$ and $\delta_1>0$
such that
\begin{equation} \label{s1def}
\mathcal U'(s), \mathcal V'(s) \geq 0 \quad\text{on $(-\infty,s_1]$}, \qquad
\mathcal U'(s) < 0 \quad\text{on $(s_1, s_1+\delta_1)$.}
\end{equation}
Combining the definition of $s_1$ with \eqref{EF}, we now find that
\begin{equation} \label{s1}
0 \geq \mathcal U''(s_1) = \mathcal U''(s_1) + (n-2-2\al) \, \mathcal U'(s_1)
= Q(\al) \, (\mathcal U(s_1) - \mathcal V(s_1)^p).
\end{equation}
This implies $s_1\notin S_{\mathcal U}$ and thus $s_1 \in S_{\mathcal V}$.  Since $S_{\mathcal V}$
is open, we have $(s_1-\delta_2, s_1+\delta_2) \subset S_{\mathcal V}$ for some small enough $0<
\delta_2 \leq \delta_1$.  It follows by \eqref{EF} and \eqref{SUSV} that $e^{(n-2-2\be)s} \mathcal
V'(s)$ is strictly increasing on this interval, hence $\mathcal U'(s) < 0 < \mathcal V'(s)$
throughout $(s_1,s_1+\delta_2)$.

Consider the largest interval $I= (s_1,s_2)$ on which $\mathcal U'(s) < 0 < \mathcal V'(s)$. Along
this interval, it is clear that $\mathcal V(s) - \mathcal U(s)^q$ is strictly increasing. Since
$s_1\in S_{\mathcal V}$, this expression is positive when $s=s_1$, so it is positive throughout
$I$, and thus $I \subset S_{\mathcal V}$.  It follows by \eqref{EF} and \eqref{SUSV} that
$e^{(n-2-2\be)s} \mathcal V'(s)$ is strictly increasing on $I$.  Next, we apply the same argument
to the expression $\mathcal U(s) - \mathcal V(s)^p$.  This is strictly decreasing on $I$, and it is
non-positive when $s=s_1$ by \eqref{s1}, so it is negative throughout $I$, and $e^{(n-2-2\al)s}
\mathcal U'(s)$ is strictly decreasing on $I$.

We now recall that $I= (s_1,s_2)$ is the largest interval on which $\mathcal U'(s) < 0 < \mathcal
V'(s)$. Since $e^{(n-2-2\be)s} \mathcal V'(s)$ increases on this interval and since
$e^{(n-2-2\al)s} \mathcal U'(s)$ decreases, we conclude that $s_2=\infty$ and that $\mathcal U(s),
\mathcal V(s)$ are eventually monotonic. Denote their limits by
\begin{equation} \label{LULV}
\mathcal U_\infty = \lim_{s\to \infty} \mathcal U(s), \qquad
\mathcal V_\infty = \lim_{s\to \infty} \mathcal V(s)
\end{equation}
and note that $0\leq \mathcal U_\infty, \mathcal V_\infty \leq 1$ by Lemma \ref{CoLe}.  If it
happens that $\mathcal U_\infty > \mathcal V_\infty^p$, then
\begin{equation*}
\lim_{s\to\infty} \, \bigl[ \mathcal U'(s) + (n-2-2\al) \, \mathcal U(s) \bigr] = +\infty
\end{equation*}
by \eqref{EF}, and thus $\lim_{s\to \infty} \mathcal U(s) = +\infty$, a contradiction.  If it
happens that $\mathcal U_\infty < \mathcal V_\infty^p$, we get a similar contradiction.  We
conclude that $\mathcal U_\infty = \mathcal V_\infty^p$ and also $\mathcal V_\infty = \mathcal
U_\infty^q$ for similar reasons. Since $pq>1$ by assumption, the only possible solutions are then
\begin{equation*}
\mathcal U_\infty = \mathcal V_\infty = 0 \quad\text{and}\quad
\mathcal U_\infty = \mathcal V_\infty = 1.
\end{equation*}
The former case is excluded because $0<\mathcal V(s)<1$ is eventually increasing, while the latter
case is excluded because $0<\mathcal U(s)<1$ is eventually decreasing.

We may thus conclude that neither $\mathcal U'(s)$ nor $\mathcal V'(s)$ can become negative. In
particular, both $\mathcal U(s)$ and $\mathcal V(s)$ are increasing for all $s$, and the limits in
\eqref{LULV} exist. Using the exact same argument as before, one finds that both limits must be
equal to $1$.
\end{proof_of}

\begin{theorem} \label{AsEx}
Let $(u,v)$ be a regular positive radial solution of the elliptic system \eqref{esys}, where
$k,l\geq 0$, $p,q\geq 1$, $pq>1$ and $n\geq 11$.  Let $(C_\al r^{-\al}, C_\be r^{-\be})$ denote the
singular solution \eqref{ss}. Assume the existence condition \eqref{Sob} and the Joseph-Lundgren
condition \eqref{JL}.

\begin{itemize}
\item[(a)] If strict inequality holds in \eqref{JL}, then there exist constants $D_1,D_2$ such that
\begin{align*}
u(r) = C_\al r^{-\al} + D_1 r^{-\al-\gamma} + o(r^{-\al-\gamma}) \quad\text{as $r\to \infty$,} \\
v(r) = C_\be r^{-\be} + D_2 r^{-\be-\gamma} + o(r^{-\be-\gamma}) \quad\text{as $r\to \infty$.}
\end{align*}

\item[(b)] If equality holds in \eqref{JL}, then there exist constants $D_1,D_2$ such that
\begin{align*}
u(r) = C_\al r^{-\al} + D_1 r^{-\al-\gamma} \log r + o(r^{-\al-\gamma} \log r)
\quad\text{as $r\to \infty$,} \\
v(r) = C_\be r^{-\be} + D_2 r^{-\be-\gamma} \log r + o(r^{-\be-\gamma} \log r)
\quad\text{as $r\to \infty$.}
\end{align*}
\end{itemize}
In each case, $\gamma>0$ is an explicit constant which is defined in Lemmas \ref{Roots} and
\ref{LinSol}.
\end{theorem}

\begin{remark}
We note that the first terms in the asymptotic expansions are given by the singular solution
\eqref{ss}, while the second terms are related to the linearized solution \eqref{PhiPsi}. Although
these expansions were known in the scalar case, we are not aware of any similar results in the case
of systems. For the second-order equation $-\Delta u = u^p$ which corresponds to the case $p=q>1$,
we refer the reader to \cite{GNW, LiYi}.  For the biharmonic equation $\Delta^2 u = u^p$ which
corresponds to the case $p>q=1$, we refer the reader to \cite{PK12, MW}.
\end{remark}

\begin{proof}
We only prove part (a), as part (b) is quite similar.  Consider the variables
\begin{equation} \label{yz}
y(s) = \mathcal U(s) - 1 = \frac{u(e^s)}{C_\al e^{-\al s}} - 1, \qquad
z(s) = \mathcal V(s) - 1 = \frac{v(e^s)}{C_\be e^{-\be s}} - 1,
\end{equation}
where $s= \log r$ and $\mathcal U(s), \mathcal V(s)$ are defined by \eqref{UV}.  It trivially
follows by \eqref{EF} that
\begin{equation} \label{Nonl}
\left\{ \begin{aligned}
y''(s) + (n-2-2\al) \, y'(s) - Q(\al) y(s) &= Q(\al) \, (1 - (z(s) + 1)^p) \\
z''(s) + (n-2-2\be) \, z'(s) - Q(\be) z(s) &= Q(\be) \, (1 - (y(s) + 1)^q)
\end{aligned} \right.
\end{equation}
for all $s\in\R$.  The corresponding linearized system is then
\begin{equation*}
\left\{ \begin{aligned}
y''(s) + (n-2-2\al) \, y'(s) - Q(\al) y(s) &= -p Q(\al) z(s) \\
z''(s) + (n-2-2\be) \, z'(s) - Q(\be) z(s) &= -q Q(\be) y(s).
\end{aligned} \right.
\end{equation*}
Since $Q(\la)= \la(n-2-\la)$ for all $\la\in\R$, this can be written in the compact form
\begin{equation*}
\left\{ \begin{aligned}
Q( \al - \partial_s ) \, y(s) &= p Q(\al) z(s) \\
Q( \be - \partial_s ) \, z(s) &= q Q(\be) y(s),
\end{aligned} \right.
\end{equation*}
where $\partial_s = \frac{\partial}{\partial s}$.  In particular, the associated characteristic
polynomial has the form
\begin{equation*}
F(\la) = Q(\al - \la) Q(\be - \la) - pq Q(\al) Q(\be).
\end{equation*}
Focusing on part (a), we assume that the Joseph-Lundgren condition \eqref{JL} holds with strict
inequality. It then follows by Lemma \ref{Roots} that $F(\la)$ has four distinct real roots
\begin{equation} \label{eigen}
-\la_1 < -\la_2 < -\la_3 < 0 < \la_4,
\end{equation}
where $\la_i>0$ for each $i$ and $\la_3 = \gamma$ is the positive constant introduced in Lemma
\ref{LinSol}.

To analyze the nonlinear system \eqref{Nonl}, we first express it in the compact form
\begin{align} \label{Compact}
\left\{ \begin{aligned}
Q( \al - \partial_s ) \, y(s) - p Q(\al) z(s) &= f(z(s)) \\
Q( \be - \partial_s ) \, z(s) - q Q(\be) y(s) &= g(y(s)),
\end{aligned} \right.
\end{align}
where the nonlinear terms $f,g$ are defined by
\begin{equation} \label{fg}
f(z) = Q(\al) \,( (z+1)^p - 1 - pz), \qquad g(y) = Q(\be) \, ((y+1)^q - 1 - qy).
\end{equation}
The solution of this system can be obtained using variation of parameters.  Let us mostly worry
about $y(s)$, as the approach for $z(s)$ is similar.  Variation of parameters gives
\begin{align} \label{exp1}
y(s) &= \sum_{i=1}^3 \left[ A_i e^{-\la_i s}
+ B_i \int_{s_0}^s e^{-\la_i(s-\tau)} f(z(\tau)) \,d\tau
+ C_i \int_{s_0}^s e^{-\la_i(s-\tau)} g(y(\tau)) \,d\tau \right] \notag \\
&\quad + A_4 e^{\la_4 s} + B_4 \int_{s_0}^s e^{\la_4(s-\tau)} f(z(\tau)) \,d\tau
+ C_4 \int_{s_0}^s e^{\la_4(s-\tau)} g(y(\tau)) \,d\tau,
\end{align}
where $s_0\in \R$ is arbitrary, and the constants $A_i, B_i, C_i$ can all be determined explicitly.
The constants $A_i$ depend on both $s_0$ and the eigenvalues, whereas the constants $B_i,C_i$
depend on the eigenvalues only.  In view of our definition \eqref{yz}, we must have
\begin{equation} \label{limyz}
\lim_{s\to\infty} y(s) = \lim_{s\to \infty} z(s) = 0
\end{equation}
by Theorem \ref{MCT}.  Focusing on the last two integrals in \eqref{exp1}, we now write
\begin{align*}
y(s) &= \sum_{i=1}^3 \left[ A_i e^{-\la_i s}
+ B_i \int_{s_0}^s e^{-\la_i(s-\tau)} f(z(\tau)) \,d\tau
+ C_i \int_{s_0}^s e^{-\la_i(s-\tau)} g(y(\tau)) \,d\tau \right] \notag \\
&\quad - B_4 \int_s^\infty e^{-\la_4(\tau-s)} f(z(\tau)) \,d\tau
- C_4 \int_s^\infty e^{-\la_4(\tau-s)} g(y(\tau)) \,d\tau \notag \\
&\quad + e^{\la_4 s} \left[ A_4
+ B_4 \int_{s_0}^\infty e^{-\la_4 \tau} f(z(\tau)) \,d\tau
+ C_4 \int_{s_0}^\infty e^{-\la_4 \tau} g(y(\tau)) \,d\tau \right].
\end{align*}
The terms in the first two lines all approach zero as $s\to \infty$, so the same must be true for
the term in the last line.  Since $\la_4>0$, however, this simply means that
\begin{align} \label{exp2}
y(s) &= \sum_{i=1}^3 \left[ A_i e^{-\la_i s}
+ B_i \int_{s_0}^s e^{-\la_i(s-\tau)} f(z(\tau)) \,d\tau
+ C_i \int_{s_0}^s e^{-\la_i(s-\tau)} g(y(\tau)) \,d\tau \right] \notag \\
&\quad - B_4 \int_s^\infty e^{-\la_4(\tau-s)} f(z(\tau)) \,d\tau
- C_4 \int_s^\infty e^{-\la_4(\tau-s)} g(y(\tau)) \,d\tau.
\end{align}
We now use this formula to estimate $y(s)$.  Recalling our definition \eqref{fg}, we have
\begin{equation} \label{exp3}
f(z(\tau)) = o(\tau) \cdot |z(\tau)|, \qquad
g(y(\tau)) = o(\tau) \cdot |y(\tau)| \quad\text{as $\tau\to\infty$}
\end{equation}
by \eqref{limyz} since $p,q\geq 1$.  Employing this fact in \eqref{exp2}, we obtain the estimate
\begin{align*}
|y(s)| &\leq O(e^{-\la_3 s}) + C \int_{s_0}^s e^{-\la_3(s-\tau)} o(\tau) |z(\tau)| \,d\tau +
C \int_{s_0}^s e^{-\la_3(s-\tau)} o(\tau) |y(\tau)| \,d\tau \\
&\qquad + C\int_s^\infty e^{-\la_4(\tau-s)} o(\tau) |z(\tau)| \,d\tau
+ C\int_s^\infty e^{-\la_4(\tau-s)} o(\tau) |y(\tau)| \,d\tau
\end{align*}
as $s\to\infty$ because of \eqref{eigen}.  A similar estimate holds for $|z(s)|$ by symmetry. If we
denote their sum by $w(s) = |y(s)| + |z(s)|$, we may thus conclude that
\begin{align*}
w(s) &\leq O(e^{-\la_3 s}) + C \int_{s_0}^s e^{-\la_3(s-\tau)} o(\tau) w(\tau) \,d\tau +
C\int_s^\infty e^{-\la_4(\tau-s)} o(\tau) w(\tau) \,d\tau
\end{align*}
as $s\to \infty$.  This is easily seen to imply that $w(s) = O(e^{-\la_3 s})$ as $s\to \infty$,
namely
\begin{equation*}
|y(s)| + |z(s)| = w(s) = O(e^{-\la_3 s}) \quad\text{as $s\to \infty$.}
\end{equation*}
Since the functions $f,g$ defined by \eqref{fg} are quadratic near the origin, we also have
\begin{equation} \label{exp4}
f(z(\tau)) = O(e^{-2\la_3 \tau}), \qquad
g(y(\tau)) = O(e^{-2\la_3 \tau}) \quad\text{as $\tau\to\infty$.}
\end{equation}
This provides a more precise version of our previous estimate \eqref{exp3}.  Combining this new
version with our identity \eqref{exp2}, it is now easy to verify that
\begin{equation*}
y(s) = A_3 e^{-\la_3 s} + o(e^{-\la_3 s}) \quad\text{as $s\to\infty$.}
\end{equation*}
Since $s=\log r$ and $\gamma=\lambda_3$ by definition, it follows by \eqref{yz} that
\begin{equation*}
u(r) = C_\al r^{-\al} + A_3 C_\al r^{-\al-\gamma} + o(r^{-\al-\gamma}) \quad\text{as $r\to\infty$.}
\end{equation*}
A similar expansion holds for $v(r)$, and this completes the proof of part (a).  For part (b), the
exact same argument works, but $-\la_3$ is a double eigenvalue by Lemma \ref{Roots}.
\end{proof}

\section{Further properties of regular solutions}

We now establish some further properties of the regular positive radial solutions of \eqref{esys}.
Our starting point is a useful fact which is due to Serrin and Zou \cite{SZ3, SZ1}.  If $(u,v)$ is
any such solution, then its central value $(u(0),v(0))$ lies on the graph of a function $f$ which
is both continuous and strictly increasing.  This applies to general Hamiltonian systems
\cite{SZ3}. In the case of \eqref{esys}, however, the function $f$ can also be determined
explicitly.

\begin{lemma}
Consider the elliptic system \eqref{esys}, where $k,l\geq 0$, $p,q\geq 1$ and $pq>1$.  If the
existence condition \eqref{Sob} holds, then the regular positive radial solutions of the system
form a one-parameter family $\{ (u_\xi,v_\xi) \}_{\xi>0}$ such that
\begin{equation} \label{cv}
u_\xi(0) = \xi^\al, \qquad v_\xi(0) = C_0 \xi^\be
\end{equation}
for all $\xi>0$, where $C_0$ is a positive constant which only depends on $k,l,p,q$ and $n$.
\end{lemma}

\begin{proof}
We refer the reader to \cite{BVG} for the existence and \cite{HV} for the uniqueness of solutions.
Denote by $(u_1,v_1)$ the unique regular positive radial solution which satisfies $u_1(0)=1$. Then
every other regular positive radial solution can be obtained through the scaling
\begin{equation} \label{scale}
u_\xi(r) = \xi^\al u_1(\xi r), \qquad v_\xi(r) = \xi^\be v_1(\xi r).
\end{equation}
Letting $C_0 = v_1(0)$, one finds that $(u_\xi,v_\xi)$ is the unique solution which satisfies
\eqref{cv}.
\end{proof}

\begin{lemma} \label{inter}
Let the assumptions of the previous lemma hold, and assume that the Joseph-Lundgren condition
\eqref{JL} also holds.  Given any $\xi_1 > \xi_2 > 0$, we then have
\begin{equation}
u_{\xi_1}(r) > u_{\xi_2}(r) \quad\text{and}\quad
v_{\xi_1}(r) > v_{\xi_2}(r) \quad\text{for all $r\geq 0$.}
\end{equation}
\end{lemma}

\begin{proof}
According to Theorem \ref{MCT}, the functions defined by
\begin{equation*}
\mathcal{U}_\xi(r) = r^\al u_\xi(r), \qquad \mathcal{V}_\xi(r) = r^\be v_\xi(r)
\end{equation*}
are both increasing with respect to $r$.  On the other hand, it follows by \eqref{scale} that
\begin{equation*}
\mathcal{U}_\xi(r) = (\xi r)^\al u_1(\xi r), \qquad \mathcal{V}_\xi(r) = (\xi r)^\be v_1(\xi r).
\end{equation*}
Monotonicity with respect to $r$ is thus equivalent to monotonicity with respect to $\xi$, so
\begin{equation*}
u_{\xi_1}(r) \geq u_{\xi_2}(r) \quad\text{and}\quad
v_{\xi_1}(r) \geq v_{\xi_2}(r) \quad\text{for all $r\geq 0$.}
\end{equation*}
It remains to show that strict inequality holds in each case.  Since
\begin{equation} \label{ineq}
(r^{n-1} (u_{\xi_1}' - u_{\xi_2}'))' = r^{n-1+k} \, (v_{\xi_2}^p - v_{\xi_1}^p) \leq 0
\quad\text{for all $r\geq 0$,}
\end{equation}
we find that $u_{\xi_1} - u_{\xi_2}$ is decreasing to zero by Lemma \ref{Limit}. If it vanishes at
some point $R$, it must thus vanish throughout $[R,\infty)$.  In that case, \eqref{ineq} gives
$v_{\xi_1} = v_{\xi_2}$ throughout $[R,\infty)$, and this obviously contradicts the uniqueness of
solutions.  In particular, $u_{\xi_1} - u_{\xi_2}$ remains positive at all points, and the same is
similarly true for $v_{\xi_1} - v_{\xi_2}$.
\end{proof}

\begin{proof_of}{Theorem \ref{ST}}
This is a typical application of Theorem \ref{AsEx} and Lemma \ref{inter}.  In fact, one may
proceed as in the classical argument by Gui, Ni and Wang \cite{GNW}.  We only give a sketch of the
proof for part (a), as the proof of part (b) is quite similar.

To show that the solution $(u_\xi,v_\xi)$ is stable with respect to the norm \eqref{Norm1}, let
$\e>0$ be given.  According to Theorem \ref{AsEx}, there exists some $R_{\xi,\e}>0$ such that
\begin{equation*}
u_\xi(r) = C_\al r^{-\al} + D_1(\xi) r^{-\al-\gamma} + E_\xi(r)
\quad\text{in $[R_{\xi,\e},\infty)$,}
\end{equation*}
where $|E_\xi(r)| \leq \e r^{-\al-\gamma}$ in $[R_{\xi,\e},\infty)$.  On the other hand, it follows
by scaling that
\begin{equation*}
u_\eta(r) = \frac{\eta}{\xi} \,u_\xi \bigl( (\eta/\xi)^{1/\al} \, r \bigr)
\quad\text{for each $\eta>0$.}
\end{equation*}
Assume that $\eta$ is sufficiently close to $\xi$, say $|\eta - \xi| < \frac{\xi}{2}$. Combining
the last two equations, we can then find constants $R_\e, C$ independent of $\eta$ such that
\begin{equation*}
u_\eta(r) = C_\al r^{-\al} + D_1(\eta) r^{-\al-\gamma} + E_\eta(r)
\quad\text{in $[R_\e, \infty)$,}
\end{equation*}
where $D_1(\eta) = (\frac{\xi}{\eta})^{\al+\gamma} D_1(\xi)$ and $|E_\eta(r)|\leq C\e
r^{-\al-\gamma}$ in $[R_\e, \infty)$.  This implies that
\begin{equation*}
|r^{\al+\gamma} (u_\eta - u_\xi)| \leq |D_1(\eta) - D_1(\xi)| + 2C\e
\quad\text{in $[R_\e, \infty)$,}
\end{equation*}
and we also have $|(1+r)^{\al+\gamma} (u_\eta - u_\xi)| \to 0$ uniformly in $[0,R_\e]$ as $\eta\to
\xi$. Since similar estimates hold with $u_\eta$ and $u_\xi$ replaced by $v_\eta$ and $v_\xi$,
respectively, it follows by \eqref{Norm1} that
\begin{equation*}
\lim_{\eta \to \xi} || (u_\eta, v_\eta) - (u_\xi, v_\xi) ||_\gamma = 0.
\end{equation*}
In particular, given any $\e>0$, there exists some small enough $\theta\in (0, \frac{\xi}{2} )$
such that
\begin{equation} \label{stab1}
|| (u_{\xi\pm\theta}, v_{\xi\pm\theta}) - (u_\xi, v_\xi) ||_\gamma < \e.
\end{equation}

We now exploit this fact in order to prove stability.  According to Lemma \ref{inter}, one has
\begin{equation*}
u_{\xi -\theta} < u_\xi < u_{\xi + \theta} \quad\text{and}\quad
v_{\xi -\theta} < v_\xi < v_{\xi + \theta}
\end{equation*}
at all points.  Arguing as in \cite{GNW}, we can thus find some small enough $\delta>0$ such that
\begin{equation*}
u_{\xi -\theta} < f < u_{\xi + \theta} \quad\text{and}\quad
v_{\xi -\theta} < g < v_{\xi + \theta}
\end{equation*}
whenever $||(f,g) - (u_\xi,v_\xi)||_\gamma < \delta$.  Let $(u,v)$ be the solution of the parabolic
system \eqref{psys} subject to $u(x,0)=f(x)$ and $v(x,0)=g(x)$.  Since $(u_{\xi\pm\theta},
v_{\xi\pm\theta})$ satisfy the same system, one may use standard comparison theorems for parabolic
systems \cite{LS} to conclude that
\begin{equation} \label{stab2}
u_{\xi -\theta} < u < u_{\xi + \theta} \quad\text{and}\quad
v_{\xi -\theta} < v < v_{\xi + \theta},
\end{equation}
as long as $||(f,g) - (u_\xi,v_\xi)||_\gamma < \delta$.  The result now follows by \eqref{stab1}
and \eqref{stab2}.
\end{proof_of}

\section*{Acknowledgements}
The first author acknowledges the financial support of The Irish Research Council Postgraduate
Scholarship under grant number GOIPG/2022/469.

\end{document}